\documentclass[10pt]{amsart}
\usepackage{amssymb,amsthm,amsmath,amsfonts,accents}
\usepackage[numbers]{natbib}
\usepackage{hyperref}
\usepackage{enumerate}

\makeatletter
\g@addto@macro\th@plain{\thm@headpunct{}}
\makeatother

\providecommand{\scalar}[1]{\left\langle#1\right\rangle}
\newcommand{\uhat}{\underaccent{\check}}
\newcommand{\E}{\mathbb{E}}
\newcommand{\R}{\mathbb{R}}

\numberwithin{equation}{section}
\theoremstyle{plain}
\newtheorem{theorem}{Theorem}[section]
\newtheorem{lemma}{Lemma}[section]
\newtheorem{remark}{Remark}[section]
\newtheorem{example}{Example}[section]

\newtheorem{definition}{Definition}[section]



\providecommand{\scalar}[1]{\left\langle#1\right\rangle}

\begin{document}

\title[Riesz Exponential Family on homogeneous cones]{Characterization of the Riesz Exponential Family on homogeneous cones}

\author[H. Ishi]{Hideyuki Ishi}
\author[B. Ko\l{}odziejek]{Bartosz Ko\l{}odziejek}

\address{Nagoya University, Graduate School of Mathematics, Furo-cho, Nagoya 464-8602, Japan \slash 
	{JST PRESTO, 4-1-8 Honcho, Kawaguchi 332-0012, Japan.}}
\email{hideyuki@math.nahoya-u.ac.jp}

\address{Warsaw University of Technology, Faculty of Mathematics and Information Science \\Koszykowa 75, 00-662 Warsaw, Poland}
\email{b.kolodziejek@mini.pw.edu.pl}

\begin{abstract}
In the paper we present a characterization theorem of {the} Riesz measure {and a Wishart exponential family} on homogeneous cones through the invariance property of a natural exponential family under the action of the triangular group.
\end{abstract}

\maketitle

\section{Introduction}

Following Casalis \cite{Cas91}, we consider natural exponential family (NEF), which is invariant under a subgroup $G$ of the general linear group of {a} finite dimensional linear space. NEF is considered on arbitrary finite dimensional linear space $\E$. Under a weak assumption on $G$ (see Theorem \ref{th2}), {the} generating measure of such family is of the form 
$$e^{-\scalar{\theta_0,x}}\mu_0(dx)$$
for some $\theta_0$ in the dual space $\E^\ast$, where $\mu_0$ is a $G$-invariant measure. 

We apply Theorem \ref{th2} to {the} problem of a characterization of the Riesz measure on {a} homogeneous cone through the invariance property of NEF under the action of the triangular group {and we solve this problem in Theorem \ref{TH35}}. Since NEF generated by the Riesz measure consists of Wishart distributions, {Theorem \ref{TH35}} {provides} also a characterization of a Wishart distribution on homogeneous cones \cite{AW04,GI14,LM07}.
We would like to mention that this problem was announced by Letac in \cite[Section 4]{Le89}, who pointed out that the natural framework for such characterizations is indeed homogeneous cone. 

There are essentially two types of characterizations of NEFs. The first one is connected with the central object of NEF, that is, the variance function. The aim is then to describe a generating measure of a NEF with given variance function. Much {has} been done in this direction, but still much more is unknown. A generating measure is known for $\E=\R$, when the variance function is a quadratic polynomial \cite{Mo82}, cubic polynomial \cite{HZ04,LM90} and of some more sophisticated forms, like $P\Delta+Q\sqrt{\Delta}$, where $P$, $\Delta$ and $Q$ are quadratic polynomials \cite{Le92}. For $\E=\R^d$ we have to emphasize \cite{Ca91} with its deep connection between homogeneous quadratic variance functions and Euclidean Jordan algebras, simple quadratic \cite{Ca96} and simple cubic variance functions \cite{HaZa06}, just to name a few.

The second type of characterizations of NEFs {is} through the invariance properties under some group action. We quote some results on invariant NEFs under a given subgroup $G$ of the general affine group: one parameter group \cite{Er84}, group of rotations \cite{Wa83}, {Moebius} group \cite{Je81},
{the identity component $\mathrm{Aut}_0(\mathrm{Sym}_+(N,\R))$ of the linear automorphism group of the cone
	$\mathrm{Sym}_+(N,\R) \subset \mathrm{Sym}(N,\R)$ \cite{Le89},} 
triangular group of a simple Euclidean Jordan algebra \cite{HaLa01} and its modification \cite{HaLa04}. Here $\mathrm{Sym}(N,\R)$ stands for the symmetric $N\times N$ matrices with real entries and $\mathrm{Sym}_+(N,\R)$ is the cone of positive definite real $N\times N$ matrices.
In the last three quoted papers the characterizations were carried out by showing that the variance function (which uniquely determines NEF) coincides with the variance function of some Riesz measure (or its image by the involution $x\mapsto-x$) on $\mathrm{Sym}_+(N,\R)$ {(\cite{Le89})} and symmetric cones {(\cite{HaLa01,HaLa04})}. 

In the present paper we {solve the problem raised by Letac \cite{Le89} on homogeneous cones}. 
We use a matrix realization of homogeneous cones, which proves here very useful and is more accessible to the reader who is not familiar with homogeneous cones, clans and $T$-algebras.
{We would like to emphasize that homogeneous cones are of great importance in statistics despite their theoretical character. 
	Particularly, homogeneous cones appear very naturally in statistics in the context of graphical Gaussian models \cite{LM07}. The aforementioned matrix realization of homogeneous cones allows to describe many of the colored graphical Gaussian models \cite{SS08}. 
	The formula for variance function of $F({\mathcal{R}_{\underline{s}}})$ on homogeneous cones is the topic of our joint paper with Piotr Graczyk \cite{GIK18}.}
At last we would like to point out that this work is closely related to \cite{Sh80,Sh07}, where NEFs and information theory are treated on homogeneous Hessian structures.

The paper is organized as follows. The concept of natural exponential families is introduced in the next section. In Section \ref{Sec52} we give an application of Theorem \ref{th2} to a characterization of the Riesz measure on homogeneous cone {(Theorem \ref{TH35})}. A crucial role in the proof of characterization is played by a matrix realization of any homogeneous cone (\cite{Hi06}), which is explained in Section \ref{Sec51}.
Section \ref{Sec5} {ends with} some comments.

\section{Natural Exponential Families}
In {this} section we will introduce all the necessary facts on NEFs that will be needed later on. It is however worth mentioning that the standard reference book on exponential families is \cite{BN14}.

Let $\E$ be a finite dimensional real linear space 
and $\E^\ast$ its dual space. 
{The coupling of $\theta \in E^*$ and $x \in E$ is denoted by 
	$\scalar{\theta,x}$.}
Let $\mu$ be a positive Radon measure on $\E$. We define its Laplace transform $L_\mu\colon \E^\ast\to(0,\infty]$ by
$$L_\mu(\theta)=\int_{\E} e^{\scalar{\theta,x}}\mu(dx).$$
By $\Theta(\mu)$ we will denote the interior of the set $\{\theta\in \E^\ast\colon L_\mu(\theta)<\infty\}$. H\"older's inequality implies that the set $\Theta(\mu)$ is convex and the cumulant function
$$k_\mu(\theta)=\log L_\mu(\theta)$$
is convex on $\Theta(\mu)$ and it is strictly convex if and only if $\mu$ is not concentrated on some affine hyperplane of $\E$.
Let $\mathcal{M}(\E)$ be the set of positive Radon measures on $\E$ such that $\Theta(\mu)$ is not empty and $\mu$ is not concentrated on some affine hyperplane of $\E$. 

For $\mu\in\mathcal{M}(\E)$ we define the \emph{natural exponential family (NEF) generated by $\mu$} (denoted by $F(\mu)$) as the set of probability measures of the form
$$P(\theta,\mu)(dx)=e^{\scalar{\theta,x}-k_\mu(\theta)}\mu(dx),\quad \theta\in\Theta(\mu).$$

Let us note that $F(\mu)=F(\mu')$ if and only if 
$\mu'(dx)=e^{\scalar{a,x}+b}\mu(dx)$ for some $a\in\E^\ast$ and $b\in\R$.

We will now describe the action of elements from the general linear group $GL(\E)$ on a NEF. The identity element of $GL(\E)$ will be denoted by $\mathrm{Id}$. 
Let $F=F(\mu)$ be a NEF on $\E$. {Let $g_\ast\mu$ denote the image measure of $\mu$ by $g$ and let $g.F(\mu)$ stand for the family of image measures 
	$g_\ast P(\theta,\mu)\,\,\,(\theta \in \Theta(\mu))$.}
Then, for any $g\in GL(\E)$, we have ${g.F(\mu)}=F(g_\ast\mu)$.

We say that a measure $\mu_0$ is \emph{invariant under a subgroup $G$} of $GL(\E)$ if for all $g\in G$ there exists a constant $c_g>0$ for which $\mu_0(gA)=c_g\mu_0(A)$ for any measurable set $A\subset \E$. This condition is equivalent to
\begin{align}\label{iL}
	L_{\mu_0}(g^\ast\theta)=c_g^{-1}L_{\mu_0}(\theta),\quad \theta\in \Theta(\mu_0).
\end{align}
{Note that the correspondence $G\owns g \mapsto c_g \in \R$ is a character of the group $G$.}

Let $\mu\in\mathcal{M}(\E)$. Observe that the condition $g.F(\mu)=F(\mu)$ implies that for any $g\in G$ there exists $a(g)\in\Theta(\mu)\subset \E^\ast$ and $b(g)\in\R$ such that
\begin{align}\label{eqq}
	g_\ast\mu(dx)=e^{\scalar{a(g),x}+b(g)}\mu(dx).
\end{align}
Then, Casalis \cite{Cas91} showed that (see \cite[Theorem 2.2]{Cas91}) functions $a$ and $b$ satisfy the following system of equations: for any $(g,g')\in G^2$ 
\begin{align*}
	a(g g')&=(g^\ast)^{-1}a(g')+a(g), \\
	b(g g')&=b(g)+b(g').
\end{align*}
Let us assume that $G$ contains $c \mathrm{Id}$ for some $c\neq 1$. Then, we obtain for any $g\in G$,
\begin{align*}
	a(c\mathrm{Id}\, g)&=\frac{1}{c}a(g)+a(c\mathrm{Id}), \\
	a(g\, c\mathrm{Id})&=(g^\ast)^{-1}a(c\mathrm{Id})+a(g). \\
\end{align*}
Equating the right hand sides of above formulas, we arrive at
$$a(g)=\theta_0-(g^\ast)^{-1}\theta_0,\qquad g\in G,$$
with $\theta_0=\frac{c}{1-c}{a}(c\mathrm{Id})$.
Define $\mu_0(dx)=e^{\scalar{\theta_0,x}}\mu(dx)$. Then, \eqref{eqq} implies that $\mu_0$ is $G$-invariant. 
Thus, we obtain the following
\begin{theorem}\label{th2}
	Let $G$ be a subgroup of $GL(\E)$ and let $F=F(\mu)$ be a $G$-invariant NEF on $\E$, that is, for any $g\in G$ one has $g{.}F=F$. If $G$ contains $c \mathrm{Id}$ for some $c\neq 1$, then there exist $\theta_0\in \E^\ast$ and a $G$-invariant measure $\mu_0$ such that
	$$\mu(dx)=e^{-\scalar{\theta_0,x}}\mu_0(dx).$$
	{In such case we have also $F=F(\mu_0)$.}
\end{theorem}

\section{Characterization of the Riesz measure on homogeneous cones}\label{Sec5}
\subsection{Matrix realization of homogeneous cones}\label{Sec51}
Let $V$ be a real linear space and $\Omega$ a regular open convex set in $V$ containing no line. The cone $\Omega$ is said to be \emph{homogeneous} if the linear automorphism group $G(\Omega)=\{g\in GL(V)\colon g\Omega=\Omega \}$ acts transitively on $\Omega$, that is, for any $x$ and $y$ in $\Omega$ there exists $g\in G(\Omega)$ such that $y=gx$.

We will now give a very useful representation of homogeneous cones following \cite[Section 3]{Is14}.
	For a symmetric matrix $x\in\mathrm{Sym}(N,\R)$, we denote by $\uhat{x}$ the lower triangular matrix of size $N$ defined by
	\begin{align*}
		(\uhat{x})_{ij}=\begin{cases}
			x_{ij}, & \mbox{if }i>j,\\
			x_{ii}/2,&\mbox{if }i=j,\\
			0,&\mbox{if }i<j.
		\end{cases}
	\end{align*}
	Then we have $x=\uhat{x}+\hat{x}$, where $\hat{x}=\uhat{x}^\top$ is the transpose of $\uhat{x}$. For $x,y\in\mathrm{Sym}(N,\R)$, we define
	$$x\bigtriangleup y:=\uhat{x}y+y\hat{x}\in\mathrm{Sym}(N,\R).$$
	Then, $(\mathrm{Sym}(N,\R),\bigtriangleup)$ forms a non-associative algebra with unit element $I_N$. Let $\mathcal{Z}$ be a subalgebra of $(\mathrm{Sym}(N,\R),\bigtriangleup)$ and $H_\mathcal{Z}$ be the set of lower triangular matrices from {elements of }$\mathcal{Z}$ with positive diagonal entries, that is, $H_{\mathcal{Z}}:=\{ \uhat{x}\colon x\in\mathcal{Z}\mbox{ and }x_{ii}>0\}$.\\
	Define $\Omega_\mathcal{Z}=\{x\in \mathcal{Z}\colon {x}\mbox{ is positive definite}\}$ and consider for any $T\in H_{\mathcal{Z}}$ the linear operators $\rho(T)\colon \mathcal{Z}\to\mathcal{Z}, x\mapsto \rho(T)x=TxT^\top$. It can be shown that $\rho(H_{\mathcal{Z}})$ acts on $\Omega_\mathcal{Z}$ transitively, which means that $\Omega_\mathcal{Z}$ is a homogeneous cone (\cite[Theorem 3]{Is15}). Furthermore, we have the following
	\begin{theorem}[\cite{Is11}]\label{TH31}
		For a homogeneous cone $\Omega\subset V$, there exists a subalgebra $\mathcal{Z}\subset\mathrm{Sym}(N,\R)$ and a linear isomorphism $\phi\colon V\to\mathcal{Z}$ such that $\phi(\Omega)=\Omega_{\mathcal{Z}}$.
\end{theorem}

We say that $\mathcal{Z}\subset\mathrm{Sym}(N,\R)$ admits a normal block decomposition if there exists a partition $N=n_1+\ldots+n_r$ and subspaces $\mathcal{V}_{lk}\subset\mathrm{Mat}(n_l,n_k,\R)$, $1\leq k<l\leq r$, such that $\mathcal{Z}$ is the set of symmetric matrices of the form
\begin{align*}
	\begin{pmatrix}
		X_{11} & X_{21}^\top & \cdots & X_{r1}^\top \\
		X_{21} & X_{22} & & X_{r2}^\top \\
		\vdots & & \ddots & \\
		X_{r1} & X_{r2} & & X_{rr}
	\end{pmatrix}
	\qquad 
	\begin{pmatrix}
		X_{ll}=x_{ll}I_{n_l},\,\, x_{ll}\in\R,\,\,1\leq l\leq r \\
		X_{lk}\in \mathcal{V}_{lk}, \,\,1\leq k<l\leq r 
	\end{pmatrix}.
\end{align*}
We will write $\mathcal{Z}_\mathcal{V}$ for this space.
\begin{theorem}
	\begin{itemize}
		\item[(i)] (\cite[Theorem 2]{Is15})
		Let $\mathcal{Z}$ be a subalgebra of $(\mathrm{Sym}(N,\R),\bigtriangleup)$  with $I_N\in\mathcal{Z}$. Then there exists a permutation matrix $w$, such that $w\mathcal{Z}w^\top$ admits a normal block decomposition.
		\item[(ii)] (\cite[Proposition 2]{Is15})
		$\mathcal{Z}_\mathcal{V}$ is a subalgebra of $(\mathrm{Sym}(N,\R),\bigtriangleup)$ if and only if the subspaces $\{\mathcal{V}_{lk}\}_{1\leq k<l\leq r}$ satisfy the following conditions:
		\begin{itemize}
			\item[{\rm (V1)}] $A\in\mathcal{V}_{lk}$, $B\in\mathcal{V}_{ki}$ $\implies$ $AB\in\mathcal{V}_{li}$ for any $1\leq i<k<l\leq r$,
			\item[{\rm (V2)}] $A\in\mathcal{V}_{li}$, $B\in\mathcal{V}_{ki}$ $\implies$ $AB^\top\in\mathcal{V}_{lk}$ for any $1\leq i<k<l\leq r$,
			\item[{\rm (V3)}] $A\in\mathcal{V}_{lk}$ $\implies$ $AA^\top\in\R I_{n_l}$ for any $1\leq k<l\leq r$.
		\end{itemize}
	\end{itemize}
\end{theorem}
If $\mathcal{Z}=\mathcal{Z}_\mathcal{V}$ we shall write $\Omega_\mathcal{V}$, $H_\mathcal{V}$ for $\Omega_{\mathcal{Z}}$ and $H_{\mathcal{Z}}$, respectively.

Condition $(V3)$ allows us to define an inner product on $\mathcal{V}_{lk}$, $1\leq k<l\leq r$, by
$$AA^\top=(A|A)I_{n_l},\qquad A\in\mathcal{V}_{lk}.$$
We then define the \emph{standard inner product} on $\mathcal{Z}_\mathcal{V}$ by
$$\scalar{x,y}:=\sum_{k=1}^r x_{kk}y_{kk}+2\sum_{1\leq k<l\leq r} (X_{lk}|Y_{lk}),\qquad x,y\in\mathcal{Z}_\mathcal{V}.$$
{We identify $\mathcal{Z}_\mathcal{V}^\ast$ with $\mathcal{Z}_\mathcal{V}$ using the standard inner product.} Note that $\scalar{\cdot,\cdot}$ coincides with the trace inner product only if $n_1=\ldots=n_r=1$.

Define a one-dimensional representation of $H_{\mathcal{V}}$ by
$$\chi_{\underline{s}}(T):=\prod_{k=1}^r t_{kk}^{2 s_k},$$ 
where $\underline{s}=(s_1,\ldots,s_r)\in\mathbb{C}^r$. 
{We have (see e.g. \cite[Lemma 2.4]{GG01})}
\begin{lemma}\label{lemREP}
	Let $\chi$ be a one-dimensional representation of $H_{\mathcal{V}}$. Then, there exists $\underline{s}\in\mathbb{C}^r$ such that
	$$\chi=\chi_{\underline{s}}.$$
\end{lemma}
This fact will be important later on.

For any open convex cone $\Omega$ we define the \emph{dual cone} of $\Omega$ by
$$\Omega^\ast=\{\xi\in V^\ast\colon \scalar{{\xi,}x}>0 \,\,\,\forall x\in \bar{\Omega}\setminus\{0\}\},$$
where $V^\ast$ is the dual space of $V$. If $\Omega$ is homogeneous, then so is $\Omega^\ast$. Let $\Omega_\mathcal{V}^\ast$ denote the dual cone of $\Omega_\mathcal{V}$. 

For $T \in H_{\mathcal{V}}$, we denote by $\rho^\ast(T)$ the adjoint operator of $\rho(T) \in GL(\mathcal{Z}_{\mathcal{V}})$
defined in such a way that
$\scalar{\rho^\ast(T)\xi{, x}} = \scalar{{\xi, }\rho(T)x}$ for $x,\,\xi \in \mathcal{Z}_{\mathcal{V}}$.
Then we see from \cite[Chapter 1, Proposition 9]{Vi63} that for any $\xi\in \Omega_\mathcal{V}^\ast$, there exists a unique $T\in H_\mathcal{V}$ such that $\xi=\rho^\ast(T)I_N$. 
\begin{definition}
	Let $\Delta^\ast_{\underline{s}}\colon \Omega_\mathcal{V}^\ast\to\mathbb{C}$ be the function given by
	$$\Delta^\ast_{\underline{s}}(\xi)=\Delta^\ast_{\underline{s}}(\rho^\ast(T)I_N):=\chi_{\underline{s}^\ast}(T),$$
	where $\underline{s}^\ast=(s_r,\ldots,s_1)\in\mathbb{C}^n$.
\end{definition}

\begin{example}
	Let
	$$\mathcal{Z}_\mathcal{V}:=\left\{ \begin{pmatrix}x_{11} & 0 & x_{31} \\ 0 & x_{22} & x_{32} \\ x_{31} & x_{32} & x_{33} \end{pmatrix}\colon x_{11},x_{22},x_{33},x_{31},x_{32}\in\R\right\}.$$
	Conditions {\rm (V1)--(V3)} are satisfied and we have $n_1=n_2=n_3=1$, $N=r=3$.
	Then,
	$$\Omega_\mathcal{V}=\mathcal{Z}_\mathcal{V}\cap \mathrm{Sym}_+(3,\R)=\left\{ x\in\mathcal{Z}_\mathcal{V}\colon x_{11}>0,\,x_{22}>0,\,\det x>0\in\R\right\}$$
	and its dual cone is given by 
	$$\Omega_\mathcal{V}^\ast=\left\{ \xi\in \mathcal{Z}_\mathcal{V} \colon \xi_{33}>0,\,\xi_{11}\xi_{33}>\xi_{31}^2,\,\xi_{22}\xi_{33}>\xi_{32}^2\right\}.$$ 
	The cone $\Omega_\mathcal{V}^\ast$
	is called the Vinberg cone, while $\Omega_\mathcal{V}$ is called the dual Vinberg cone.
	The cones $\Omega_\mathcal{V}^\ast$ and $\Omega_\mathcal{V}$  are the lowest dimensional non-symmetric homogeneous cones.
	Moreover, for $\xi\in \Omega_\mathcal{V}^\ast$, we have 
	$$\Delta^\ast_{\underline{s}}(\xi)=\left(\frac{\xi_{11}\xi_{33}-\xi_{31}^2}{\xi_{33}} \right)^{s_3}\left(\frac{\xi_{22}\xi_{33}-\xi_{32}^2}{\xi_{33}} \right)^{s_2}\xi_{33}^{s_1}.$$
\end{example}

We see that for any $S\in H_\mathcal{V}$ and $\xi=\rho^\ast(T)I_N\in\Omega^\ast_{\mathcal{V}}$ we have
\begin{align}\label{propDelta}
	\Delta^\ast_{\underline{s}}(\rho^\ast(S)\xi)=\chi_{\underline{s}^\ast}(TS)=\chi_{\underline{s}^\ast}(T)\chi_{\underline{s}^\ast}(S)=\Delta^\ast_{\underline{s}}(\xi)\Delta^\ast_{\underline{s}}(\rho^\ast(S)I_N)
\end{align}
and, since any character of $H_\mathcal{V}$ is of the form $\chi_{\underline{s}^\ast}$, property {\eqref{propDelta}} characterizes  $\Delta^\ast$ {(see \cite{BK15})}. Function $\Delta^\ast$ is sometimes termed a generalized power function.
Its importance is emphasized by the following result \cite{Gi75,Is00}.
\begin{theorem}\label{riesz}
	There exists a positive measure $\mathcal{R}_{\underline{s}}$ on $\mathcal{Z}_\mathcal{V}$ with the Laplace transform $L_{\mathcal{R}_{\underline{s}}}(-\theta)=\Delta^\ast_{-\underline{s}^\ast}(\theta)$ for $\theta\in\Omega_\mathcal{V}^\ast$ if and only if $\underline{s}\in \Xi:=\bigsqcup_{\underline{\varepsilon}\in\{0,1\}^r} \Xi(\underline{\varepsilon})$,
	where
	$$\Xi(\underline{\varepsilon}):=\begin{Bmatrix}
	& s_k>p_k(\underline{\varepsilon})/2\mbox{ if } \varepsilon_k=1 \\
	\underline{s}\in\R^r; & \\
	& s_k=p_k(\underline{\varepsilon})/2\mbox{ if } \varepsilon_k=0
	\end{Bmatrix}$$
	and $p_k(\underline{\varepsilon})\:=\sum_{i<k} \varepsilon_i\dim\mathcal{V}_{ki}$.
\end{theorem}
The measure $\mathcal{R}_{\underline{s}}$ with $\underline{s}\in\Xi$ has the support in $\overline{\Omega_\mathcal{V}}$ and is called the \emph{Riesz measure}, while the set $\Xi$ is called the \emph{Gindikin-Wallach set}. 

In order to define NEF generated by $\mathcal{R}_{\underline{s}}$ we have to know if $\mathcal{R}_{\underline{s}}\in\mathcal{M}(\mathcal{Z}_\mathcal{V})$, that is, if $\mathcal{R}_{\underline{s}}$ is not concentrated on some affine hyperplane of $\mathcal{Z}_\mathcal{V}$. The following result is a generalization of \cite[Theorem 3.1]{HaLa01}.
Let us define for $\underline{\varepsilon}\in\{-1,0,1\}^r$,
$$E_{\underline{\varepsilon}}:=\begin{pmatrix}
\varepsilon_1 I_{n_1} & & \\ & \ddots & \\ & & \varepsilon_r I_{n_r}
\end{pmatrix}\in\mathcal{Z}_\mathcal{V}$$

\begin{theorem}\label{hi}
	The support of $\mathcal{R}_{\underline{s}}$ is not concentrated on any affine hyperplane in $\mathcal{Z}_{\mathcal{V}}$ 
	if and only if $s_k >0$ for all $k=1, \dots, r$.
\end{theorem}
\begin{proof}
	We write $\mathcal{O}_{\underline{\varepsilon}}$ for the $\rho(H_{\mathcal{V}})$-orbit in $\mathcal{Z}_{\mathcal{V}}$ through $E_{\underline{\varepsilon}}$.
	Note that $\mathcal{O}_{(1, \dots, 1)} = \Omega_{\mathcal{V}}$.
	It is shown in \cite[Theorem 6.2]{Is00} that,
	if $\underline{s} \in \Xi(\underline{\varepsilon})$, then $\mathcal{R}_{\underline{s}}$ is a positive measure on $\mathcal{O}_{\underline{\varepsilon}}$,
	so that the support of $\mathcal{R}_{\underline{s}}$ coincides with the closure $\overline{\mathcal{O}_{\underline{\varepsilon}}}$ of $\mathcal{O}_{\underline{\varepsilon}}$.
	In particular, if for any $k=1, \dots, r$,
	\begin{equation} \label{eqn:regular_measure}
		s_k > p_k(1, \dots, 1)/2 = \frac{1}{2} \sum_{i<k} \dim \mathcal{V}_{ki},
	\end{equation}
	then $\underline{s} \in \Xi(1, \dots, 1)$ and
	$\mathcal{R}_{\underline{s}}$ is a regular measure on the cone $\Omega_{\mathcal{V}}=\rho(H_\mathcal{V})I_N$.
	
	Now we show the `if' part of the statement.
	Assume that $\underline{s} \in \Xi \cap \R^r_{>0}$.
	In view of (\ref{eqn:regular_measure}),
	we see that there exists a positive integer $m$ such that $m \underline{s} \in \Xi(1, \dots, 1)$.
	Then $\mathcal{R}_{m \underline{s}}$ is a regular measure,
	while $\mathcal{R}_{m \underline{s}}$ equals the convolution measure $\mathcal{R}_{\underline{s}} * \mathcal{R}_{\underline{s}} * \dots * \mathcal{R}_{\underline{s}}$ ($m$ times).
	It follows that the support of $\mathcal{R}_{\underline{s}}$ is not concentrated on any affine hyperplane in $\mathcal{Z}_{\mathcal{V}}$.
	
	Next we show the `only if' part. 
	It suffices to show that,
	if $\underline{s} \in \Xi(\underline{\varepsilon})$ with $s_k = 0$ for some $k$,
	then $\mathrm{supp}\,\mathcal{R}_{\underline{s}} = \overline{\mathcal{O}_{\underline{\varepsilon}}}$ is contained in 
	the subspace $(\R E_k)^{\perp} := \{ x \in \mathcal{Z}_{\mathcal{V}} \colon x_{kk} = 0\}$ of $\mathcal{Z}_{\mathcal{V}}$.
	Recalling the definition of $\Xi(\underline{\varepsilon})$, we see that
	$$
	\varepsilon_k = 0 \quad\mbox{ and }\quad p_k(\underline{\varepsilon}) = \sum_{i<k} \varepsilon_i \dim \mathcal{V}_{ki} = 0,
	$$ 
	and the latter equality implies
	$$
	\mathcal{V}_{ki} = {\{0\}} \mbox{ if }\varepsilon_i = 1.
	$$
	Therefore,
	for any $x = \rho(T) E_{\underline{\varepsilon}} = T E_{\underline{\varepsilon}} T^{\top} \in \mathcal{O}_{{\underline{\varepsilon}}}$ with $T \in H_{\mathcal{V}}$, 
	we have
	$$
	x_{kk} = \varepsilon_k (t_{kk})^2 + \sum_{i<k} \varepsilon_i \Vert T_{ki} \Vert^2 = 0,
	$$
	which means that $\mathcal{O}_{\underline{\varepsilon}} \subset (\R E_k)^{\perp}$.
	Hence $\mathrm{supp}\,\mathcal{R}_{\underline{s}} \subset (\R E_k)^{\perp}$ and the proof is completed.   
\end{proof}

For $\underline{\varepsilon}\in\{-1,1\}^r$, consider
$\mathcal{O}^\ast_{\underline{\varepsilon}}:=\rho^\ast(H_{\mathcal{V}})E_{\underline{\varepsilon}}$.
The set
\begin{align}\label{setd}
	\bigsqcup_{\underline{\varepsilon}\in\{-1,1\}^r} \mathcal{O}^\ast_{\underline{\varepsilon}}
\end{align}
is dense in $\mathcal{Z}_\mathcal{V}$ and $\mathcal{O}^\ast_{\underline{\varepsilon}}$ are the only open orbits of $\rho^\ast(H_{\mathcal{V}})$ (see 
{\cite[p.~77]{Gi64}}).

\subsection{Characterization of the Riesz measure on homogeneous cone}\label{Sec52}
In the following section we will give an application of Theorem \ref{th2} to a characterization of the Riesz measure on homogeneous cone. We generalize the results of \cite{HaLa01}, where the characterization of Riesz measure through invariance property of NEF on simple Euclidean algebras was considered. 

We say that the subalgebra $\mathcal{Z}$ of $(\mathrm{Sym}(N,\R),\bigtriangleup)$ is \emph{irreducible} if $\mathcal{Z}$ is not equal to a direct sum of two non-trivial ideals. 
\begin{theorem}\label{TH35}
	Let $\E=\mathcal{Z}_\mathcal{V}$ be {an irreducible subalgebra}{ of $(\mathrm{Sym}(N,\R),\bigtriangleup)$ that admits a normal block decomposition} and {let $\mu\in\mathcal{M}(\E)$}. Assume that $F(\mu)$ {is} a NEF invariant by $G=\rho(H_\mathcal{V})$.
	Then there exist $\theta_0\in\mathcal{Z}_\mathcal{V}$, {$a_0 \in \R$} and $\underline{s}\in\Xi \cap \R^r_{>0}$ such that
	$$\mu(dx)=e^{{a_0} -\scalar{\theta_0,x}} \mathcal{R}_{\underline{s}}(dx)$$
	or 
	$$\mu(dx)=e^{{a_0}- \scalar{\theta_0,x}} \mathcal{R}_{\underline{s}}(-dx)$$
\end{theorem}
Thanks to Theorem \ref{TH31}, Theorem \ref{TH35} provides a characterization of Riesz measures on homogeneous cones (note that the support of characterized measure is a homogeneous cone). 

\begin{proof}
	We have $\R_{>0}I_N\subset H_\mathcal{V}$ and so $c\,\mathrm{Id}\in\rho(H_\mathcal{V})$ for all $c>0$.
	Theorem \ref{th2} implies that there exists $\theta_0$ such that
	$$\mu(dx)=e^{-\scalar{\theta_0,x}}\mu_0(dx),$$
	where {(by \eqref{iL})}
	\begin{align}\label{eqLap}
		L_{\mu_0}(\rho^\ast(T)\theta)=\chi(T)L_{\mu_0}(\theta),\quad (\theta,T)\in \Theta(\mu_0)\times H_\mathcal{V}
	\end{align}
	{with a certain character $\chi$ of $H_\mathcal{V}$.}
	We will determine $L_{\mu_0}$ and $\Theta(\mu_0)$. The proof is split into four steps:
	\begin{itemize}
		\item[(i)] There exists $\underline{\varepsilon}\in\{-1,1\}^r$ such that $E_{\underline{\varepsilon}}\in \Theta(\mu_0)$. Moreover, if $E_{\underline{\varepsilon}},E_{\underline{\varepsilon}'}\in \Theta(\mu_0)$, then $\underline{\varepsilon}=\underline{\varepsilon}'$.
		\item[(ii)] $\Theta(\mu_0)=\mathcal{O}^\ast_{\underline{\varepsilon}}$ for some $\underline{\varepsilon}\in\{-1,1\}^r$.
		\item[(iii)] If $\Theta(\mu_0)=\mathcal{O}^\ast_{\underline{\varepsilon}}$ then $\underline{\varepsilon}=(-1,\ldots,-1)$ or $\underline{\varepsilon}=(1,\ldots,1)$.
		\item[({iv})] $\mu$ is of the postulated form.
	\end{itemize}
	\emph{First step.} By definition, the set $\Theta(\mu_0)$ is open and non-empty. Since $\rho^\ast(T)\Theta(\mu_0)=\Theta(\mu_0)$ for any $T\in H_\mathcal{V}$ and the set \eqref{setd} is dense in $\mathcal{Z}_\mathcal{V}$, we have $E_{\underline{\varepsilon}}\in\mathcal{O}^\ast_{\underline{\varepsilon}}\subset \Theta(\mu_0)$ for some $\underline{\varepsilon}\in\{-1,1\}^r$. \\
	Assume now that $E_{\underline{\varepsilon}}, E_{\underline{\varepsilon}'}\in \Theta(\mu_0)$ and $\underline{\varepsilon}\neq\underline{\varepsilon}'$.
	Since $\Theta(\mu_0)$ is convex, we know that 
	$\sigma:=\frac12(E_{\underline{\varepsilon}}+E_{\underline{\varepsilon'}})\in\Theta(\mu_0)$. Define $I_0:=\{i\in\{1,\ldots,r\}\colon \frac{\varepsilon_i+\varepsilon_i'}{2}=0\}$. The set $I_0$ is not empty.
	Define 
	$$H_0:=\left\{T\in H_\mathcal{V}\colon t_{ii}=1\mbox{ for any }i\notin I_0 \right\}$$
	If $T\in H_0$ is diagonal, then 
	$$\rho^\ast(T)\sigma=\sigma.$$
	On the other hand, by \eqref{eqLap} we obtain
	$$L_{\mu_0}(\rho^\ast(T)\sigma)=\chi(T)L_{\mu_0}(\sigma),$$
	which implies that $\chi(T)=1$ for any $T\in H_0$ (note that $\chi(T)$ depends on $T$ only through diagonal elements{; see Lemma \ref{lemREP}}).
	Further, this implies that for any {diagonal} $T=\mathrm{diag}(t_{11},\ldots,t_{rr})\in H_0$, 
	$$L_{\mu_0}(\rho^\ast(T) E_{\underline{\varepsilon}})=L_{\mu_0}(E_{\underline{\varepsilon}}),$$
	{where $\rho^\ast(T) E_{\underline{\varepsilon}}=\mathrm{diag}\left(t_{11}\varepsilon_1,\ldots,t_{rr}\varepsilon_r\right)$.}
	{T}his contradicts the assumption that $\mu_0$ is not supported on any affine hyperplane of $\mathcal{Z}_\mathcal{V}$. 
	Indeed, {function $(0,\infty)\ni t_{ii}\mapsto \int \exp(t_{ii}\varepsilon_ix_{ii})\mu_0(dx)$ is constant if and only if the support of $\mu_0$ is contained in the subspace of $\mathcal{Z}_\mathcal{V}$ whose $(i,i)$-components are zero for any $i\in I_0$.
	}
	
	\emph{Second step.} Since $\Theta(\mu_0)$ contains only one open orbit, we have
	$$\mathrm{int}\{\Theta(\mu_0)\setminus \mathcal{O}^\ast_{\underline{\varepsilon}}\}=\emptyset.$$
	This implies that
	$$\Theta(\mu_0)\setminus \overline{\mathcal{O}^\ast_{\underline{\varepsilon}}}=\emptyset,$$
	thus
	$$\mathcal{O}^\ast_{\underline{\varepsilon}}\subset \Theta(\mu_0)\subset \overline{\mathcal{O}^\ast_{\underline{\varepsilon}}}$$
	which proves the claim, since $\Theta(\mu_0)$ is open.
	
	\emph{Third step.} We will show that $\mathcal{O}^\ast_{\underline{\varepsilon}}$ is not convex unless all $\varepsilon_k$, $1\leq k\leq r$, are simultaneously $1$ or $-1$.
	Define
	$$I^+(\underline{\varepsilon}):=\{i\in\{1,\ldots,r\}\colon \varepsilon_i=1\}\,\mbox{ and }\,I^-(\underline{\varepsilon}):=\{1,\ldots,r\}\setminus I^+(\underline{\varepsilon})$$
	Suppose that $I^+(\underline{\varepsilon})$ and $I^-(\underline{\varepsilon})$ are non-empty. 
	Then there exists $k\in I^+(\underline{\varepsilon})$ and $l\in I^-(\underline{\varepsilon})$ such that the space $\mathcal{V}_{lk}$ is not {$\{0\}$} (without loss of generality we may assume that $l>k$). If not,  $\mathcal{Z}_\mathcal{V}$ has the block form
	$$w\left(\begin{array}{c|c}
	\mathcal{Z}_+ & 0 \\
	\hline
	0 & \mathcal{Z}_-
	\end{array}\right)w^\top,
	$$
	for some permutation matrix $w$.
	This contradicts the assumption that $\mathcal{Z}_\mathcal{V}$ is irreducible. 
	
	Thus there exists $\mathcal{V}_{lk}\ne {\{0\}}$. 
	We have $\varepsilon_k=1$ and $\varepsilon_l=-1$. For $v\in \mathcal{V}_{lk}$ let $T(v)$ be the element of $H_{\mathcal{V}}$ such that $T_{lk}{(v)}=v$, $t_{ii}{(v)}=1$ for $i=1,\ldots r$ and $T_{ji}{(v)}=0$ for all $(i,j)\neq (k,l)$, $1\leq i<j\leq r$. Take any $v\in\mathcal{V}_{lk}$ with $(v|v)=2$. Then
	\begin{align}\label{nonc}
		\frac12\left(\rho^\ast(T(v)){E_{\underline{\varepsilon}}}+\rho^\ast(T(-v))
		{E_{\underline{\varepsilon}}}\right)=E_{\underline{\varepsilon}'},
	\end{align}
	where $\varepsilon'_i=\varepsilon_i$ for $i\neq k$ and $\varepsilon'_k=-1$. 
	We will use the definition of $\rho^\ast(T)$, {but there is other natural approach; see Remark \ref{refnew} after the proof}. For any $x\in\mathcal{Z}_\mathcal{V}$ consider the matrix
	$$x_v:=\frac{1}{2}\left(\rho(T(v))x+\rho(T(-v))x\right)=\frac{1}{2}\left(T(v)xT(v)^\top+{T(-v)}xT(-v)^\top\right).$$
	It may be verified by direct calculation that matrices $x$ and $x_v$ differ only on their $(l,l)$-components, which in the latter case equals
	$$( v v^\top x_{kk} + x_{ll})I_{n_l}=(2x_{kk} + x_{ll})I_{n_l}.$$
	Thus,
	\begin{align*}
		\scalar{\frac12\left(\rho^\ast(T(v)) {E_{\underline{\varepsilon}}}+\rho^\ast(T(-v)){E_{\underline{\varepsilon}}} \right){, x}}=\scalar{{E_{\underline{\varepsilon}},} \frac{1}{2}\left(\rho(T(v))x+\rho(T(-v))x\right) } \\
		= \sum_{i\notin \{k,l\}} x_{ii}\varepsilon_i+{x_{kk}}+(2x_{kk} + x_{ll})(-1)=\scalar{E_{\underline{\varepsilon}'}{, x}}.
	\end{align*}
	But $\Theta(\mu_0)$ is convex, thus {\eqref{nonc} implies that} $E_{\underline{\varepsilon}'}\in\Theta(\mu_0)$. This contradicts the point $\mathrm{(i)}$.
	This means that {$I^+(\underline{\varepsilon})=\{1,\ldots,r\}$ or $I^-(\underline{\varepsilon})=\{1,\ldots,r\}$} and so $-I_N$ or $I_N$ belongs to $\Theta(\mu_0)${. F}inally, {by \rm{(ii)} and the fact that $\rho^\ast(H_\mathcal{V})I_N=\Omega^\ast_{\mathcal{V}}$},
	$$\Theta(\mu_0)=-\Omega^\ast_\mathcal{V}\quad\mbox{ or }\quad\Theta(\mu_0)=\Omega^\ast_\mathcal{V}.$$
	
	\emph{Fourth step.}
	Let us first consider the case $\Theta(\mu_0)=-\Omega^\ast_\mathcal{V}$. Putting $\theta=\rho^\ast(S)(-I_N)$ for $S\in H_\mathcal{V}$, we obtain
	\begin{align*}
		L_{\mu_0}(-\rho^\ast(ST) I_N)=\chi(T)L_{\mu_0}(-\rho^\ast(S)I_N),\quad (S,T)\in  H_{\mathcal{V}}^2,
	\end{align*}
	which implies that for ${a_0= \log L_{\mu_0}(-I_N)}$,
	$$L_{\mu_0}(- \rho^\ast(T)I_N)={e^{a_0}}\chi(T)$$
	and $\chi(T)$ is a one-dimensional representation of  $H_\mathcal{V}$. Thus{, by Lemma \ref{lemREP},} there exists $\underline{s}\in \Xi$ such that
	$\chi=\chi_{-\underline{s}}$ and then{, by the definition of $\Delta^\ast_{\underline{s}^\ast}$,}
	$$L_{\mu_0}(-\theta)={e^{a_0}}\Delta^\ast_{{-}\underline{s}^\ast}(\theta),\quad \theta\in\Omega^\ast_\mathcal{V}{=-\Theta(\mu_0)}.$$
	By Theorems \ref{riesz} and \ref{hi} we see that $\underline{s}\in\Xi \cap \R^r_{>0}${ and }that $\mu_0(dx)={e^{a_0}}\mathcal{R}_{\underline{s}}(dx)$ and 
	$\Theta(\mu_0)=-\Omega^\ast_\mathcal{V}$.
	
	In the second case, when $\Theta(\mu_0)=\Omega^\ast_\mathcal{V}$, one show{s similarly} that $\mu_0(dx)={e^{a_0}}\mathcal{R}_{\underline{s}}(-dx)$. 
\end{proof}

\begin{remark}\label{refnew}
	{To show \eqref{nonc} we could switch to a matrix realization of the dual cone $\Omega_\mathcal{V}^\ast$, where (under suitable linear isomorphism) $\rho^\ast(T)$ is just a multiplication by some upper triangular matrix on the left and its transpose on the right.}
\end{remark}

\subsection{Comments}
\begin{enumerate}
	\item In \cite{HaLa04}, the authors considered characterization of the Riesz measure $\mathcal{R}_{\underline{s}}$ through the invariance property of NEF on a simple Euclidean algebra $\E$ by some subgroup $G$ of $GL(\E)$. This subgroup was carefully chosen in order to ensure that some components of vector $\underline{s}\in\Xi$ are equal. Taking ordinary triangular group $\rho(H_\mathcal{V})$ imposes no additional conditions on these components. On the other hand if one considers the invariance property of NEF by the connected component containing identity of $G(\Omega)$ (this is in principle what Letac did on symmetric matrices in \cite{Le89}, but it is true on homogeneous cones also; see Comment {\rm (4)} below), then all $s_i$ {have} to be equal, $s_i=p$ for $1\leq i\leq r$. Then $\Delta^\ast_{\underline{s}^\ast}=\det^p$ and $p$ belongs to the set $\Lambda$ called the Jorgensen set (see \cite{CaLe94}).
	\item Elements of $F({\mathcal{R}_{\underline{s}}})$ are actually the Wishart distributions on homogeneous cones introduced in \cite{AW04} (the subcase of $s_i=p$ for $1\leq i\leq r$) and in \cite{GI14}.
	\item It should be stressed that our approach is very different {from} \cite{HaLa01,Le89}, where the characterization of NEF was proved by showing that the variance function of $\rho(H_\mathcal{V})$-invariant NEF coincides with the one of $F({\mathcal{R}_{\underline{s}}})$ for some $\underline{s}\in\Xi$. {In the present paper} we {do not} even need to know what is the variance function of the Riesz measure on homogeneous cones. We perceive our approach as less technical and more natural.
	\item {We can rephrase our result in terms of theory of homogeneous cones as follows. Let $\Omega \subset \E$ be an irreducible homogeneous cone, and $G \subset GL(\E)$ a linear algebraic group acting on $\Omega$ transitively. Then we conclude that any $G$-invariant NEF is generated by a Riesz distributions on $\Omega$ or $-\Omega$ because $G$ contains a triangular subgroup isomorphic to the group $\rho(H_{\mathcal{V}})$ discussed in the present paper (cf. \cite[Chapter 1, Section 9]{Vi63}).}  
\end{enumerate}

\subsection*{Acknowledgements} 
We would like to thank the anonymous referee for helpful comments.
B. Ko{\l}odziejek was partially supported by NCN Grant No. 2012/05/B/ST1/00554.
{H. Ishi was partially supported by JSPS KAKENHI Grant Number 16K05174 and JST PRESTO.}



\normalsize


\begin{thebibliography}{HD82}



\bibitem{AW04}
S.~A. Andersson and G.~G. Wojnar.
\newblock Wishart distributions on homogeneous cones.
\newblock {\em J. Theoret. Probab.}, 17(4):781--818, 2004.

\bibitem{BN14}
O.~Barndorff-Nielsen.
\newblock {\em Information and exponential families in statistical theory}.
\newblock Wiley Series in Probability and Statistics. John Wiley \& Sons, Ltd.,
Chichester, 2014.

\bibitem{Cas91}
M.~Casalis.
\newblock Familles exponentielles naturelles sur $\mathbb{R}^d$ invariantes par
un groupe.
\newblock {\em Int. Stat. Rev.}, 59(2):241--262, 1991.

\bibitem{Ca91}
M.~Casalis.
\newblock Les familles exponentielles \`a variance quadratique homog\`ene sont
des lois de {W}ishart sur un c\^one sym\'etrique.
\newblock {\em C. R. Acad. Sci. Paris S\'er. I Math.}, 312(7):537--540, 1991.

\bibitem{Ca96}
M.~Casalis.
\newblock The {$2d+4$} simple quadratic natural exponential families on {${\bf
		R}\sp d$}.
\newblock {\em Ann. Statist.}, 24(4):1828--1854, 1996.

\bibitem{CaLe94}
M.~Casalis and G.~Letac.
\newblock Characterization of the {J}\o rgensen set in generalized linear
models.
\newblock {\em Test}, 3(1):145--162, 1994.

\bibitem{Er84}
P.~S. Eriksen.
\newblock {$(k,1)$} exponential transformation models.
\newblock {\em Scand. J. Statist.}, 11(3):129--145, 1984.

\bibitem{GG01}
G.~Garrig\'os.
\newblock Generalized {H}ardy spaces on tube domains over cones.
\newblock {\em Colloq. Math.}, 90(2):213--251, 2001.

\bibitem{Gi64}
S.~G. Gindikin.
\newblock Analysis in homogeneous domains.
\newblock {\em Russian Math. Surveys}, 19:1--89, 1964.

\bibitem{Gi75}
S.~G. Gindikin.
\newblock Invariant generalized functions in homogeneous domains.
\newblock {\em Funkcional. Anal. i Prilo\v zen.}, 9(1):56--58, 1975.

\bibitem{GI14}
P.~Graczyk and H.~Ishi.
\newblock Riesz measures and {W}ishart laws associated to quadratic maps.
\newblock {\em J. Math. Soc. Japan}, 66(1):317--348, 2014.

\bibitem{GIK18}
P.~Graczyk, H.~Ishi and B.~Ko{\l}odziejek.
\newblock Wishart laws and variance function on homogeneous cones.
\newblock to appear in {\em Probab. Math. Statist.}, pages 1--25, 2019.

\bibitem{HaLa01}
A.~Hassairi and S.~Lajmi.
\newblock Riesz exponential families on symmetric cones.
\newblock {\em J. Theoret. Probab.}, 14(4):927--948, 2001.

\bibitem{HaLa04}
A.~Hassairi and S.~Lajmi.
\newblock Classification of {R}iesz exponential families on a symmetric cone by
invariance properties.
\newblock {\em J. Theoret. Probab.}, 17(3):521--539, 2004.

\bibitem{HZ04}
A.~Hassairi and M.~Zarai.
\newblock Characterization of the cubic exponential families by orthogonality
of polynomials.
\newblock {\em Ann. Probab.}, 32(3B):2463--2476, 2004.

\bibitem{HaZa06}
A.~Hassairi and M.~Zarai.
\newblock Characterization of the simple cubic multivariate exponential
families.
\newblock {\em J. Funct. Anal.}, 235(1):69--89, 2006.

\bibitem{SS08}
S.~H\o jsgaard and S.~L. Lauritzen.
\newblock Graphical {G}aussian models with edge and vertex symmetries.
\newblock {\em J. R. Stat. Soc. Ser. B Stat. Methodol}, 70(5):1005--1027, 2008.

\bibitem{Is00}
H.~Ishi.
\newblock Positive {R}iesz distributions on homogeneous cones.
\newblock {\em J. Math. Soc. Japan}, 52(1):161--186, 2000.


\bibitem{Hi06}
H.~Ishi.
\newblock On symplectic representations of normal {$j$}-algebras and their
application to {X}u's realizations of {S}iegel domains.
\newblock {\em Differential Geom. Appl.}, 24(6):588--612, 2006.

\bibitem{Is11}
H.~Ishi.
\newblock Representation of clans and homogeneous cones.
\newblock {\em Vestn. Tambov. Univ.}, 16:1669--1675, 2011.

\bibitem{Is14}
H.~Ishi.
\newblock Homogeneous cones and their applications to statistics.
\newblock In {\em Modern methods of multivariate statistics}, volume~82, pages
135--154. Hermann, 2014.

\bibitem{Is15}
H.~Ishi.
\newblock Matrix realization of a homogeneous cone.
\newblock In {\em Geometric Science of Information}, volume 9389 of {\em
	Lecture Notes in Computer Science}, pages 248--256. Springer International
Publishing, 2015.

\bibitem{Je81}
J.~L. Jensen.
\newblock On the hyperboloid distribution.
\newblock {\em Scand. J. Statist.}, 8(4):193--206, 1981.

\bibitem{BK15}
B.~Ko\l odziejek.
\newblock Multiplicative {C}auchy functional equation on symmetric cones.
\newblock {\em Aequationes Math.}, 89(4):1075--1094, 2015.

\bibitem{Le89}
G.~Letac.
\newblock A characterization of the {W}ishart exponential families by an
invariance property.
\newblock {\em J. Theoret. Probab.}, 2(1):71--86, 1989.

\bibitem{Le92}
G.~Letac.
\newblock {\em Lectures on natural exponential families and their variance
	functions}, volume~50 of {\em Monograf\'\i as de Matem\'atica [Mathematical
	Monographs]}.
\newblock Instituto de Matem\'atica Pura e Aplicada (IMPA), Rio de Janeiro,
1992.

\bibitem{LM07}
G.~Letac and H.~Massam.
\newblock Wishart distributions for decomposable graphs.
\newblock {\em Ann. Statist.}, 35(3):1278--1323, 2007.

\bibitem{LM90}
G.~Letac and M.~Mora.
\newblock Natural real exponential families with cubic variance functions.
\newblock {\em Ann. Statist.}, 18(1):1--37, 1990.

\bibitem{Mo82}
C.~N. Morris.
\newblock Natural exponential families with quadratic variance functions.
\newblock {\em Ann. Statist.}, 10(1):65--80, 1982.

\bibitem{Sh80}
H.~Shima.
\newblock Homogeneous {H}essian manifolds.
\newblock {\em Ann. Inst. Fourier (Grenoble)}, 30(3):91--128, 1980.

\bibitem{Sh07}
H.~Shima.
\newblock {\em The geometry of {H}essian structures}.
\newblock World Scientific Publishing Co. Pte. Ltd., Hackensack, NJ, 2007.

\bibitem{Vi63}
\`E.~B. Vinberg.
\newblock The theory of homogeneous convex cones.
\newblock {\em Tr. Mosk. Mat. Obs.}, 12:303--358, 1963.

\bibitem{Wa83}
G.~S. Watson.
\newblock {\em Statistics on spheres}.
\newblock University of Arkansas Lecture Notes in the Mathematical Sciences, 6.
John Wiley \& Sons, Inc., New York, 1983.
\newblock A Wiley-Interscience Publication.

\end{thebibliography}
\end{document}